\documentclass[a4paper,12pt]{article}
\usepackage{amsthm}
\usepackage{amsmath}
\usepackage{latexsym}
\usepackage{url}
\usepackage{color}
\usepackage{comment}
\numberwithin{equation}{section}

\def\R{{\bf R}}

\def\N{{\bf N}}

\def\d{\displaystyle}
\def\e{{\varepsilon}}

\def\wt{\widetilde}

\def\p{\partial}

\newcommand{\dt}{\partial_{t}}%

\newcommand{\ga}{\gamma}

\newcommand{\si}{\sigma}

\newcommand{\I}{\infty}

\newcommand{\EQS}[1]{\begin{align} #1 \end{align}}

\newtheorem{thm}{Theorem}[section]

\newtheorem{lem}{Lemma}[section]
\newtheorem{prop}{Proposition}[section]
\newtheorem{rem}{Remark}[section]
\newtheorem{definition}{Definition}[section]

\title{Small data blow-up for the wave equation with a time-dependent scale invariant damping \\and a cubic convolution \\for slowly decaying initial data}
\author{
Masahiro Ikeda
\footnote{Department of Mathematics, Faculty of Science and Technology, Keio University, 3-14-1 Hiyoshi,
Kohoku-ku, Yokohama, 223-8522, Japan/Center for Advanced Intelligence Project, RIKEN, Japan. email: masahiro.ikeda@keio.jp/masahiro.ikeda@riken.jp.}
,
Tomoyuki Tanaka
\footnote{Graduate School of Mathematics, Nagoya University, Chikusa-ku, Nagoya, 464-8602, Japan/Center for Advanced Intelligence Project, RIKEN, Japan. e-mail: d18003s@math.nagoya-u.ac.jp/tomoyuki.tanaka.hw@riken.jp.}
\ and \
Kyouhei Wakasa
\footnote{Department of Creative Engineering, National Institute of Technology, Kushiro College, 2-32-1 Otanoshike-Nishi, Kushiro-Shi, Hokkaido 084-0916, Japan. e-mail: wakasa@kushiro-ct.ac.jp.
}
}
\date{
\[
\begin{array}{llll}
\mbox{\footnotesize{\bf Keywords:}}
& \mbox{\footnotesize Wave equation, Time-dependent scale invariant damping,}\\
& \mbox{\footnotesize Cubic convolution, Small data blow-up,}\\
& \mbox{\footnotesize Slowly decaying initial data}\\
\mbox{\footnotesize{\bf MSC2010:}}
& \mbox{\footnotesize Primary 35B44 Secondary, 35L05, 35L71}\\
\end{array}
\]
}
\begin{document}
\maketitle
\begin{abstract}
In the present paper, we study the Cauchy problem for the wave equation with a time-dependent scale invariant damping, i.e.$\frac{2}{1+t}\dt v$ and a cubic convolution $(|x|^{-\ga}*v^2)v$ with $\ga\in (0,n)$, where $v=v(x,t)$ is an unknown function on $\R^n\times[0,T)$. Our aim of the present paper is to prove a small data blow-up result and show an upper estimate of lifespan of the problem for slowly decaying positive initial data $(v(x,0),\partial_tv(x,0))$ such as $\partial_tv(x,0)=O(|x|^{-(1+\nu)})$ as $|x|\rightarrow\infty$. Here $\nu$ belongs to the scaling supercritical case $\nu<\frac{n-\gamma}{2}$. The proof of our main result is based on the combination of the arguments in the papers \cite{TUW10} and \cite{T95}. Especially, our main new contribution is to estimate the convolution term in high spatial dimensions, i.e. $n\ge 4$. This paper is the first blow-up result to treat wave equations with the cubic convolution in high spatial dimensions ($n\ge 4$).
\end{abstract}

\tableofcontents


\section{Introduction}
\subsection{Setting of our problem and its background}
\ \ In the present paper, we study the Cauchy problem
for the wave equation with a time-dependent scale invariant damping and a cubic convolution:
\EQS{\label{IVP}
\begin{cases}
  \dt^2 v-\Delta v+\displaystyle{\frac{\mu}{1+t}}\dt v=(V_{\gamma}*v^2)v, &(x,t)\in\R^n\times [0,T),\\
  v(x,0)=\e f(x), &x\in\R^n,\\
  \dt v(x,0)=\e g(x), &x\in\R^n.
\end{cases}
}
Here $n\in \N$ denotes the spatial dimension, $T=T(\varepsilon)\in (0,\infty]$ denotes the maximal existence time of the function $v$, which is called lifespan, $V(x):=|x|^{-\gamma}$ is a given function on $\R^n$ and is called the inverse power potential, where $\gamma\in (0,n)$ is a constant, $*$ stands for the convolution in the space variables, $\mu$ is a non-negative constant, $v=v(x,t)$ is an unknown function on $\R^n\times [0,T)$, $(f,g)\in C^\I(\R^n)\times C^\I(\R^n)$ is a given $\R^2$-valued function on $\R^n$, which represents the shape of the initial data, and $\e>0$ is a small parameter, which denotes the size of the initial data.

Our aim of the present paper is to prove a small data blow-up result and show an upper estimate of lifespan $T_{\varepsilon}$ for small $\varepsilon$ of the problem with slowly decaying data $(f,g)$ such as $g(x)=O(|x|^{-(1+\nu)})$ as $|x|\rightarrow\infty$ (see \ref{slow-decay-2}), where $\nu$ belongs to the scaling supercritical case (see Theorem \ref{T.1.1}). Especially, our main new contribution of the present paper is to estimate the convolution term in high space dimensions, i.e. $n\ge 4$. And our main result is the first blow-up result to treat wave equations with the cubic convolution in high space dimensions ($n\ge 4$).

In the physical context, the stationary problem corresponding to \eqref{IVP} with a mass term and the Coulomb potential ($\ga=1$)
\[
    -\Delta v+v=(|x|^{-1}*|v|^2)v,\ \ \ x\in \R^n
\]
was proposed by Hartree as a model for the helium atom. Menzala and Strauss \cite{MS82} studied the Cauchy problem of \eqref{IVP} with more general potential than the inverse power potential $|x|^{-\gamma}$ and without the dissipative term ($\mu=0$) and proved local well-posedness result and small data scattering result in the energy space $H^1(\R^n)\times L^2(\R^n)$, where $H^1(\R^n)$ denotes the usual $L^2$-based Sobolev space.

The first equation of \eqref{IVP} is invariant under the scale transformation $v\mapsto v_{\sigma}$ for $\sigma>0$ given by
\begin{equation}
\label{scale}
  v_{\sigma}(x,t):=\si^{1+\frac{n-\gamma}{2}}v(\si x,\si(1+t)-1).
\end{equation}
Therefore the damping term $\frac{\mu}{1+t}\partial_tu$ is called the scale invariant damping term and is known as a threshold between“wave-like” region and “heat-like” region.

%

\subsection{Known results}
\ \ For the undamped case $(\mu=0)$ with a replacement of the cubic convolution into the power type nonlinearity $|v|^p$ with $p>1$,
i.e.,
\EQS{\label{eq_wave1}
  \begin{cases}
    \dt^2 v-\Delta v= |v|^p, &(x,t)\in\R^n\times (-T,T),\\
    v(x,0)=\e f(x), &x\in\R^n,\\
    \dt v(x,0)=\e g(x), &x\in\R^n,
  \end{cases}
}
determining a critical exponent which divides global existence and blow-up for small solutions has been extensively studied by many authors. This problem is called the Strauss conjecture. For historical backgrounds of this conjecture and detailed estimates of lifespan $T=T_{\varepsilon}$, see Introduction in \cite{TW14} and \cite{ISW1} for example. It is well known that the critical exponent for \eqref{eq_wave1} for sufficiently rapidly decaying initial data as $|x|\rightarrow\infty$ is the Strauss exponent $p_0(n)$, which is defined by
\begin{align}
\label{strauss}
	p_{0}(n):=
	\left\{
	\begin{array}{ll}
		\infty, & (n=1),\\
		\d\frac{n+1+\sqrt{n^2+10n-7}}{2(n-1)}, & (n \geq 2),
	\end{array}
	\right.
\end{align}
and is the positive root of the quadratic equation
\begin{align}
\label{gamma-st}
	(n-1)p^2-(n+1)p-2=0.
\end{align}
In other words, small data global existence holds if $p>p_0(n)$,
and small data blow-up holds if $1<p\le p_0(n)$ for sufficiently rapidly decaying initial data as $|x|\to\I$.

Our main concern in the present paper is slowly decaying initial data as $|x|\rightarrow\infty$ such as
\begin{equation}
\label{slow-decay-1}
f(x)=O(|x|^{-\nu}),\quad g(x)=O(|x|^{-1-\nu})\quad
\mbox{as} \ |x|\rightarrow \infty,
\end{equation}
where $\nu>0$ is a positive constant and denotes the speed of the spatial decay.
In three spatial dimensions ($n=3$), Asakura \cite{Asa86} studied the problem (\ref{eq_wave1}) and showed a small data global existence if $\nu> \frac{2}{p-1}$ (scaling subcritical case) and $p>p_0(3)$. Whereas, he also proved a small data blow-up result for some radial data $(f,g)$ satisfying
\begin{equation}
\label{slow-decay-2}
f \equiv0, \quad g(x)\ge \frac{A}{(1+|x|)^{1+\nu}}
\end{equation}
with $0<\nu<\frac{2}{p-1}$ (scaling supercritical case) and $p>1$, where $A$ is a positive constant.
From his two results, we see that the critical decay exponent $\nu_c$ is $\frac{2}{p-1}$, namely
\[
    \nu_c=\nu_c(3,p):=\frac{2}{p-1}.
\]
For other related results of (\ref{eq_wave1}) with slowly decaying data, see \cite{T92, T93, T94, K93, KK95, T95, K96, KK98} for example. The critical decay exponent $\nu_c$ is related to the following scaling argument. The first equation of (\ref{eq_wave1}) is invariant under the scale transformation $v\mapsto v_{\vartheta}$ for $\vartheta>0$ given by
\[   v_{\vartheta}(x,t):={\vartheta}^{\frac{2}{p-1}}v(\vartheta x,\vartheta t).
\]
The critical decay exponent $\nu_c$ is same as the power of $\vartheta$ in the right hand side of above.


Next we recall related results for the undamped $(\mu=0)$ and the cubic convolution
case, i.e.,
\EQS{\label{eq_wave2}
  \begin{cases}
    \dt^2 v-\Delta v=(V_{\gamma}*v^2)v,&(x,t)\in\R^n\times (-T,T),\\
    v(x,0)=\e f(x), &x\in\R^n,\\
    \dt v(x,0)=\e g(x), &x\in\R^n,
  \end{cases}
}
where $n\in \N$, $V_{\gamma}(x):=|x|^{-\gamma}$ is the inverse power potential on $\R^n$ with $\gamma\in (0,n)$ and
$*$ stands for the convolution in the space variables.

Hidano \cite{H20} proved a small data scattering result for the problem (\ref{eq_wave2}) with $\gamma\in (2,\frac{5}{2})$ in three space dimensions ($n=3$) for smooth initial data decaying rapidly as $|x|\rightarrow \infty$.
On the other hand, he proved a small data blow-up result of (\ref{eq_wave2}) with $\gamma\in (0,2)$ for some positive initial data
with compact support. From his two results, a critical exponent $\gamma_c$, which devides global existence and blow-up for (\ref{eq_wave2}) with compactly supported initial data is $2$, namely,
\[
     \gamma_c=2.
\]
Tsutaya \cite{T03} studied the Cauchy problem (\ref{eq_wave2}) with the data $(f,g)$ satisfying the spatial decaying condition as $|x|\rightarrow\infty$ (\ref{slow-decay-1}). In \cite{T03}, he showed a small
data global existence for (\ref{eq_wave2}) with $\gamma\in (2,3)$ if the data $(f,g)$ satisfies (\ref{slow-decay-1}) with the scaling subcritical exponent $\left(\nu>\frac{5-\gamma}{2}\right)$. On the other hand, he proved a
small data blow-up of the problem (\ref{eq_wave2}) with $\gamma\in (0,3)$ for some data $(f,g)$ satisfying (\ref{slow-decay-2}) with the scaling supercritical exponent $\left(\frac{1}{2}<\nu<\frac{5-\gamma}{2}\right)$. From his two results in \cite{T03}, we see that the critical decay exponent $\nu_c$ is $\frac{5-\gamma}{2}$ in three
space dimensions, that is
\[
\nu_c=\nu_c(3,\gamma)=\frac{5-\gamma}{2}.
\]
This critical decay exponent $\nu_c=\nu_c(n,\gamma)$ is also related to a scale invariance of the first equation of (\ref{eq_wave2}). This situation is same as in the case of the power nonlinearity. Indeed, the first equation of (\ref{eq_wave2}) is invariant under the following scale transformation $v\mapsto v_{\sigma}$ for $\sigma>0$ given by
\begin{equation}
\label{scale2}
     v_{\sigma}(x,t):=\sigma^{1+\frac{n-\gamma}{2}}v(\sigma x,\sigma t).
\end{equation}
This transformation is similar to that of (\ref{scale}). The critical decay exponent $\nu_c$ is the power of $\sigma$ in the right hand side of (\ref{scale2}).

Kubo \cite{K04} studied the Cauchy problem (\ref{eq_wave2}) with the critical exponent $\gamma=\gamma_c=2$ and proved a small data global existence result for the data $(f,g)$ satisfying the decay condition (\ref{slow-decay-1}) with the scaling subcritical exponent $\nu\in (\frac{3}{2},2)$.

From the above results, we see that there exists a unique global solution to (\ref{eq_wave2}) with the critical case, i.e. $\ga=\ga_c=2$ for small initial data decaying rapidly as $|x|\rightarrow\infty$, whereas local solution to (\ref{eq_wave1}) with the critical exponent, i.e. $p=p_0(n)$ can not be extended globally for some positive data $(f,g)$ even if $\varepsilon$ is small and $(f,g)$ has a compact support.

We remark that Karageorgis and Tsutaya \cite{KT} reported a small data blow-up of (\ref{eq_wave2}) with the critical case, i.e. $\gamma=\gamma_c=2$ in three spatial dimensions for some data $(f,g)$ satisfying the decay condition (\ref{slow-decay-2}) with the critical decay exponent, i.e. $\nu=\nu_c=\frac{3}{2}$.

Next we recall several results for the following Cauchy problem with the scale invariant damping, i.e. $\frac{\mu}{1+t}\partial_t v$ and a power type nonlinearity, i.e. $|v|^p$;
\EQS{\label{IVP-scale}
  \begin{cases}
    \dt^2 v-\Delta v+\displaystyle{\frac{\mu}{1+t}}\dt v
      =|v|^p, & (x,t)\in\R^n\times [0,T),\\
      v(x,0)=\e f(x), &x\in\R^n,\\
      \dt v(x,0)=\e g(x), &x\in\R^n.
  \end{cases}
}
Recently, well-posedness and asymptotic behavior of solutions for the problem (\ref{IVP-scale}) have been extensively studied (see \cite{D15, Wakasugi14, LTW17, IkeSoba, ISW, TL1, TL2} for example). We only recall closely related results ($\mu=2$) to this study in the present paper. In order to study the problem (\ref{IVP-scale}) with a specific constant $\mu=2$, the Liouville transform $v\mapsto u$ given by
\begin{equation}
\label{Liou}
   u(x,t):=(1+t)^{\frac{\mu}{2}}v(x,t)
\end{equation}
is useful, where $v=v(x,t)$ is a solution to the problem \eqref{IVP-scale}. Then the transformed function $u$ satisfies the following equations:
\EQS{\label{IVP-1}
  \begin{cases}
    \dt^2 u-\Delta u+\displaystyle{\frac{\mu(2-\mu)}{4(1+t)^2}} u =\frac{|u|^p}{(1+t)^{\mu(p-1)/2}}, &(x,t)\in\R^n\times [0,T),\\
    u(x,0)=\e f(x), &x\in\R^n,\\
    \dt u(x,0)=\e\left\{\mu f(x)/2+g(x)\right\}, &x\in\R^n.
  \end{cases}
}
When $\mu=0$ or $\mu=2$, the mass term $\frac{\mu(2-\mu)}{4(1+t)^2}u$ vanishes and when $\mu=2$, the first equation of (\ref{IVP-1}) becomes the usual wave equation with a power nonlinearity $|u|^p$ with an additional time decay $\frac{1}{(1+t)^{p-1}}$, that is, the transformed function $u$ satisfies
\[
     \partial_t^2u-\Delta u=\frac{|u|^p}{(1+t)^{p-1}},\ \ \ (x,t)\in \R^n\times[0,T).
\]
In the case of $\mu=2$, it is proved in \cite{Wakasugi14, D15, DLR15} that the critical exponent $p_c=p_c(n)$ for $n=1,2,3$, which divides global existence and blow-up for small solutions for smooth initial data $(f,g)$ decaying rapidly as $|x|\rightarrow\infty$, is given by
\[
p_c(n)=\max\{ p_F(n),p_0(n+2)\}.
\]
See also \cite{LTW17, IkeSoba, ISW, TL1, TL2} for general positive $\mu$. Here $p_F=p_F(n)$ for $n\in \N$ is defined by
\begin{align}
\label{fujita}
	p_{F}(n):=1+\frac{2}{n}
\end{align}
and is called the Fujita exponent. This is the $L^1$-scaling critical exponent for the following semilinear heat (Fujita) equation
\[
\dt \theta-\Delta \theta=\theta^p,\ \ \ (x,t)\in \R^n\times [0,T),
\]
where $\theta=\theta(x,t)\ge 0$ is a positive function on $\R^n\times [0,T)$.

We turn back to the original problem (\ref{IVP}). There are no small data blow-up results about the problem (\ref{IVP}). As the first step of the study, in the present paper, we consider the case where the coefficient $\mu$ of the dissipative term is $2$, that is
\[
      \mu=2.
\]
Then in the same manner as (\ref{IVP-scale}), by using the Liouville transform $v\mapsto u$ (see \ref{Liou}) again, the transformed function $u$ satisfies the following equations:
\EQS{\label{IVP-2}
  \begin{cases}
    \dt^2 u-\Delta u=\displaystyle{\frac{(V_\ga*u^2)u}{(1+t)^2}}, &(x,t)\in\R^n\times [0,T),\\
    u(x,0)=\e f(x), &x\in\R^n,\\
    \dt u(x,0)=\e\left\{f(x)+g(x)\right\}, &x\in\R^n.
  \end{cases}
}
In this paper, we prove a small data blow-up result of (\ref{IVP-2}) with $n\ge 1$ and $\gamma\in (0,n)$ for data $(f,g)$ satisfying the spatial decay condition (\ref{slow-decay-2}) as $|x|\rightarrow\infty$ with the scaling supercritical exponent $\nu\in (0,\nu_c(n,\gamma,2))$. Here the scaling critical decay exponent $\nu_c=\nu_c(n,\gamma,\mu)$ for general $\mu$ is given by
\begin{equation}
\label{scaling}
   \nu_c=\nu_c(n,\gamma,\mu):=\frac{n+2-\mu-\gamma}{2}.
\end{equation}
The proof of the main result is based on the combination of the arguments in \cite{TUW10} and \cite{T95}. Especially, our new contribution of this paper is to estimate the convolution
term in high space dimensions $(n\ge4)$. And our main result is the first blow-up result to treat wave equations with the cubic convolution in high space dimensions ($n\ge 4$).

\section{Main Result}
\label{mainr}
\ \ In this section, we state our main result in the present paper. In the following we always assume that
\begin{equation}
\label{coeffi}
     \mu=2.
\end{equation}
Then since the original Cauchy problem (\ref{IVP}) is equivalent to the problem (\ref{IVP-2}) through the Liouville transform $v\mapsto u$, which is given by $u:=(1+t)v$ (see \ref{Liou}), we consider the latter problem (\ref{IVP-2}) below.

To state the result precisely, we introduce the definitions of solution and its lifespan and several notations.

The integral equation on $\R^n\times [0,T)$ associated with the Cauchy problem (\ref{IVP-2}) is
\begin{equation}
\label{IE_u_depend}
u(x,t)=\e u^0(x,t)+L\left((V_{\gamma}*u^2)u\right)(x,t),
\end{equation}
where the function $u^0:\R^n\times \R\rightarrow\R$ is defined by
\begin{equation}
\label{u^0}
u^0(x,t):=\p_t W(f|x,t)+W(f+g|x,t),
\end{equation}
and the integral operator $L$ on $C(\R^n\times [0,T))$ is defined by
\begin{equation}
\label{L}
L(F)(x,t):=
\int_0^tW\left(\frac{F(\cdot,s)}{(1+s)^2}
\middle|x,t-s\right)ds,
\end{equation}
where $F\in C(\R^n\times[0,T))$. Here $W$ is the solution operator to the free wave equation, which is defined by
\[
W(\phi|x,t):=
\frac{1}{(2m-1)!!}\left(\frac{1}{t}\frac{\p}{\p t}\right)^{m-1}
\left\{t^{2m-1}M(\phi|x,t)\right\}.
\]
For $m\in\N$ with $n=2m+1$ or $n=2m$, the operator $M$ is defined by
\begin{equation}
\label{M}
M(\phi|x,t):=
\left\{
\begin{array}{ll}
\d\frac{1}{\omega_{n}}\int_{|\omega|=1}\phi(x+t\omega)dS_{\omega}
&\mbox{for}\ n=2m+1,\\
\d\frac{2}{\omega_{n+1}}\int_{|\xi|\le 1}\frac{\phi(x+t\xi)}{\sqrt{1-|\xi|^2}}d\xi
&\mbox{for}\ n=2m.
\end{array}
\right.
\end{equation}
Here we denote by $\omega_n$ the Lebesgue measure of the unit sphere in $\R^n$, i.e.,
\begin{equation}
\label{units}
\omega_{n}:=\frac{2\pi^{n/2}}{\Gamma\left(n/2\right)}=
\left\{
\begin{array}{ll}
\d\frac{2(2\pi)^m}{(2m-1)!!} & \mbox{for}\ n=2m+1,\\
\d\frac{2\pi^m}{(m-1)!} & \mbox{for}\ n=2m,
\end{array}
\right.
\end{equation}
where $\Gamma:\R_{\ge 0}\rightarrow\R_{\ge 0}$ is the Gamma function defined by
\begin{equation}
\label{gamma}
\Gamma(\varrho):=\int_0^{\infty}e^{-\zeta}\zeta^{\varrho-1}d\zeta.
\end{equation}

In one spatial dimension ($n=1$), the solution operator $W$ is defined by
\begin{equation}
\label{sol-one}
      W(\phi|x,t):=\frac{1}{2}\int_{x-t}^{x+t}\phi(y)dy
\end{equation}
for $\phi\in C(\R)$.

Next we give the definition of solution and its lifespan to the Cauchy problem (\ref{IVP-2}). Here $[a]$ denotes the integral part of $a\in \R$
\begin{definition}[Solution, Lifespan]
\label{def1}
Let $T>0$,$n\in \N$ with $m=[n/2]$, $(f,g)\in C^{m+1}(\R^n)\times C^{m}(\R^n)$ and $\varepsilon>0$. We say that the function $u:\R^n\times [0,T)\rightarrow \R$ is a solution to the Cauchy problem (\ref{IVP-2}) if $u$ belongs to the class $C(\R^n\times [0,T))$ and $u$ satisfies the integral equation (\ref{IE_u_depend}). We call the maximal existence time to be lifespan, which is denoted by
\[
    T_{\varepsilon}:=\sup\left\{T\in (0,\infty]:\text{there exists a unique solution $u$ to (\ref{IVP-2}) on $\R^n\times [0,T)$}\right\}.
\]
\end{definition}
Next we introduce the scaling critical decay exponent $\nu_c$ for \eqref{IVP-2}, which is given by
\begin{equation}
\label{critical-decay}
\nu_c=\nu_c(n,\gamma):=\frac{n-\gamma}{2}.
\end{equation}
Here this is same as (\ref{scaling}) with $\mu=2$.

Now we state our main result in this paper. The following theorem means a small data blow-up result and an upper estimate of lifespan to the Cauchy problem (\ref{IVP}) with a specific coefficient $\mu=2$ and with data $(f,g)$ satisfying the spatial decay condition (\ref{slow-decay-2}) as $|x|\rightarrow\infty$ in the scaling supercritical case ($\nu<\nu_c$):
\begin{thm}[Upper estimate of lifespan for slowly decaying small data]
\label{T.1.1}
Let $n\in \N$, $\gamma\in (0,n)$, $\nu\in (0,\nu_c)$, $(f,g)\in C^1(\R^n)\times C^0(\R^n)$, $R>0$ and $A>0$. We assume that $f\equiv 0$, and $g$ is radially symmetric function if $n\ge 2$ and the estimate
\begin{equation}
\label{asm-blowup2}
g(x)\ge \frac{A}{(1+|x|)^{1+\nu}},
\end{equation}
holds for any $x\in \R^n$ with $|x|\ge R$.
Then there exist positive constants $\e_0=\e_0(A,g,\gamma,n,R)>0$ and $B=B(A,g,\gamma,n)>0$ independent of $\varepsilon$ such that the lifespan $T_\e$ defined in Definition \ref{def1} satisfies the following estimate
\begin{equation}\label{upper_lifespan}
	T_\e\le B\e^{-\frac{2}{n-\gamma -2\nu}}
\end{equation}
for $0<\e\le\e_0$.
\end{thm}

\begin{rem}
We compare our blow-up result (Theorem \ref{T.1.1}) of the present paper ($\mu=2$) to the previous results (Theorem 2.1 and Theorem 3.4 in \cite {T03}) ($\mu=0$ and $n=3$). Theorem 2.1 in \cite{T03} implies a small data global existence result for (\ref{IVP}) with $\mu=0$ and $\gamma\in (2,3)$ for the data $(f,g)$ satisfying (\ref{slow-decay-1}) with $\nu>\nu_c(3,\gamma,0)=\frac{5-\gamma}{2}$ in three spatial dimension ($n=3$). In Theorem 3.4 in \cite{T03}, Tsutaya proved the similar result to ours for (\ref{IVP}) with $\mu=0$ and $\gamma\in (0,3)$ for the data $(f,g)$ satisfying (\ref{slow-decay-2}) with $\nu<\nu_c(3,\gamma,0)=\frac{5-\gamma}{2}$ in three spatial dimension ($n=3$). Thus from Theorems 2.1 and 3.4 in \cite{T03}, we see that the critical decay exponent for \eqref{IVP-1} with $\mu=0$ in three spatial dimension $n=3$ is $\nu_c(3,\ga,0)=\frac{5-\gamma}{2}$. On the other hand, from our result (Theorem \ref{T.1.1}), we can see a shift of the spatial decay condition as $|x|\rightarrow\infty$ on the data $(f,g)$ from $\nu<\frac{5-\gamma}{2}$ to $\nu<\frac{3-\gamma}{2}$. Moreover, we prove blow-up not only in three spatial dimension but also in the other spatial dimensions.
\end{rem}

\begin{rem}
In our main result (Theorem \ref{T.1.1}), a radially symmetric assumption on the data $g$ is assumed in two dimensional case and higher dimensional case ($n\ge 2$). In fact, in two or three dimensional case ($n=2,3$), we do not have to assume the radially symmetric assumption on $g$, since the fundamental solution to the free wave equation is positive on $\R^n\times \R$ in two or three dimensional case.
\end{rem}

We explain the strategy of the proof of Theorem \ref{T.1.1}. The proof is based on an iteration argument originally developed by John \cite{J79} (see also \cite{T95, TUW10}). We divide the proof into two cases, i.e. $n\ge 2$ (Section \ref{highdime}) and $n=1$ (Section \ref{oned}). In the high dimensional case ($n\ge 2$), we use Proposition \ref{lem:positive} and estimate the solution to (\ref{IE_u_depend}) from below under the radially symmetric assumption on data. Especially, the essential part of the proof is the estimate of the convolution term in high spatial dimension ($n\ge 4$) (see Proposition \ref{prop:G-est}). In one dimensional case, we use the integral equation instead of Proposition \ref{lem:positive} and estimate the solution to (\ref{IE_u_depend}) from below.

The rest of this paper is organized as follows.
In Section \ref{highdime} and Section \ref{oned}, we give a proof of Theorem \ref{T.1.1} in high spatial dimension $n\ge 2$ and one spatial dimension $n=1$ respectively. In appendix, we give a proof of Proposition \ref{lem:positive}.

\section{Proof of Theorem \ref{T.1.1} in high spatial dimension $n\ge 2$}
\label{highdime}
\ \ In this section, we give a proof of Theorem \ref{T.1.1} in high spatial dimension $n\ge 2$.
\subsection{Useful lemmas}
\ \ First we prepare several useful lemmas in order to prove the theorem.
We state a fundamental identity for spherical means proved by John \cite{J55}.
    \begin{lem}
    \label{lm:Planewave}
    Let $n\in \N$ with $n\ge2$, $b:[0,\infty)\rightarrow \R$ be a function in $C([0,\infty))$ and $\rho>0$. Then the identity
    \begin{equation}
    \label{Planewave}
    \begin{array}{ll}
    \d \int_{|\omega|=1}b(|x+\rho \omega|)dS_\omega
    \d = 2^{3-n}\omega_{n-1}(r\rho)^{2-n}\int_{|\rho-r|}^{\rho+r}\eta b(\eta)
    h(\eta, \rho ,r)d\eta,
    \end{array}
    \end{equation}
    holds for any $\rho>0$ and $x\in \R^n$ with $r=|x|$, where $\omega_{n}$ is the area of the unit sphere in $\R^{n}$ given by $\omega_{n}:=\frac{2\pi^{\frac{n}{2}}}{\Gamma\left(n/2\right)}$ (see (\ref{units})), and $h$ is defined by
    \begin{equation}
    \label{h}
    h(\eta, \rho ,r)
    :=\{\eta^2-(\rho-r)^2\}^{\frac{n-3}{2}}\{(\rho+r)^2-\eta^2\}^{\frac{n-3}{2}}.
    \end{equation}
    \end{lem}
For the proof of this lemma, see Chapter I in \cite{J55} (see also Lemma 2.1 in \cite{K04}).

Next we state a formula for the convolution term for radially symmetric functions.
 \begin{lem}
  \label{lem:conv-est}
Let $n\in \N$ with $n\ge2$, $\gamma\in \R$, $U=U(|x|)$ be a radially symmetric function on $\R^n$. Then the identity
  \begin{equation}
  \label{conv-est}
  (V_{\gamma}*U)(x)=G_{\gamma}(U)(r),
  \end{equation}
holds for any $x\in \R^n$ with $r=|x|$, provided that both hand sides are finite on $\R^n$, where $G_{\gamma}$ is a linear operator on $[0,\infty)$ given by
\begin{equation}
\label{G}
G_{\gamma}(U)(r):=\frac{2^{3-n}\omega_{n-1}}{r^{n-2}}\int_{0}^{\infty} \rho U(\rho)
    \left\{\int_{|\rho-r|}^{\rho+r}\eta^{1-\gamma}h(\eta,\rho,r)d\eta \right\}d\rho.
\end{equation}
Here $h=h(\eta,\rho,r)$ is defined in Lemma \ref{lm:Planewave} (see (\ref{h})).
\end{lem}
\begin{proof}[Proof of Lemma \ref{lem:conv-est}]
    By the definition of the convolution and by using the polar coordinate with $y=\rho\omega$, where $\rho>0$ and $\omega\in S_{n-1}$,  we have
    \[
     (V_{\gamma}*U)(x)=\int_{\R^n}\frac{U(|y|)}{|x-y|^{\gamma}}dy
    =\int_{0}^{\infty}U(\rho)\rho^{n-1}\left(
    \int_{|\omega|=1}\frac{1}{|x-\rho \omega|^{\gamma}}dS_{\omega}\right) d\rho.
    \]
By applying the identity (\ref{Planewave}) with $b(r)=|r|^{-\gamma}$, we get the identity (\ref{conv-est}), which completes the proof of the lemma.
\end{proof}

For $T>0$, $R>0$ and $\delta>0$, we introduce the region $\Sigma$ given by
\[
\Sigma=\Sigma(T,R,\delta):=\left\{(r,t)\in (0,\infty)\times (0,T): r-t \ge \max(R, \delta t)>0\right\},
\]
where $R$ and $\delta$ are given in Theorem \ref{T.1.1} and Lemma \ref{lem:positive} respectively.

The next proposition means lower estimates of radial solutions to the wave equation (\ref{IE_u_depend}) on the region $\Sigma$, which are useful to prove Theorem \ref{T.1.1}.
\begin{prop}
\label{lem:positive}
Let $n\in \N$ with $n\ge 2$, $m:=[n/2]$, $\varepsilon>0$, $(f,g)\in C^1(\R^n)\times C^0(\R^n)$ satisfy the assumptions of Theorem \ref{T.1.1}, $T>0$ and $u=u(r,t)\in C((0,\infty)\times [0,T))$ be a radial solution to (\ref{IE_u_depend}). Then there exists a positive constant $\delta= \delta(n)$ depending only on $n$ such that the estimate
\begin{equation}
\label{positive-high}
u(r,t)>0
\end{equation}
holds for any $(r,t)\in \Sigma(T,R,\delta)$.
Moreover, the estimate
\begin{equation}
\begin{array}{llll}
\label{positive}
\d u(r,t)\ge\frac{\varepsilon}{8r^m}\int_{r-t}^{r+t}\lambda^mg(\lambda)d\lambda\\
\d \qquad \qquad +\frac{1}{8r^m}\int_{0}^{t}\frac{1}{(1+s)^2}\left\{\int_{r-t+s}^{r+t-s}\!\!\!\!
\lambda^m G_{\gamma}(u^2)(\lambda,s)u(\lambda,s)
d\lambda\right\}ds,
\end{array}
\end{equation}
holds for any $(r,t)\in \Sigma(T,R,\delta)$.
\end{prop}
This proposition can be proved in the similar manner to the proofs of Lemma 2.6 in \cite{T95} and Lemma 4.1 in \cite{TUW10}. The original idea comes from a comparison argument by Keller \cite{K57}. For convenience of the readers, we give a proof of this proposition in Appendix \ref{appendix1}.

\subsection{Iteration argument in high spatial dimension $n\ge 2$}
\ \ The proof of Theorem \ref{T.1.1} is based on an iteration argument (see \cite{J79, T95, TUW10}). To proceed the argument, we estimate the convolution term $G_{\gamma}(u(r,t))$ (see (\ref{conv-est}) for the definition of $G_{\gamma}$) in high spatial dimension $n\ge 2$ (Proposition \ref{G-est}) in this subsection, which is the most essential part of the present paper. Here $u=u(r,t)$ is a radial solution to (\ref{IE_u_depend}) on $\R^n\times [0,T)$ with the data $(f,g)$ satisfying the all assumptions of Theorem \ref{T.1.1}.

The following lemma gives the first step of the iteration argument.
\begin{lem}[First step of the iteration]
\label{lem_firs}
Under the same assumptions as Proposition \ref{lem:positive}, the estimate
  \[
  u(r,t)\ge\frac{A\e t}{8(1+r+t)^{1+\nu}}
  \]
holds for any $(r,t)\in \Sigma$.
\end{lem}

\begin{proof}[Proof of Lemma \ref{lem_firs}]
Let $(r,t)\in \Sigma$. For $s\in (0,t)$ and $\lambda\in (r-t+s,\infty)$, the estimate
\[
   \lambda-s\ge r-t\ge \max(R,\delta t)
\]
holds, which implies that the estimate $u(\lambda,s)>0$ holds. Thus by the estimate (\ref{positive}) and the assumption (\ref{asm-blowup2}) on the data $g$, the inequalities
\[
u(r,t)\ge \frac{1}{8r^m}\int_{r}^{r+t}\lambda^m\frac{A\e d\lambda}{(1+\lambda)^{1+\nu}}
  \ge\frac{A\e t}{8(1+r+t)^{1+\nu}}
\]
hold, which completes the proof of the lemma.
\end{proof}


In the following proposition (Proposition \ref{prop:G-est}), we assume that there exist positive constants $a,b,c,d$ such that the estimate
  \begin{equation}
  \label{basic-est}
  u(r,t)\ge \frac{ct^a \left\{r-t-\max(R,\delta t)\right\}^d}{(1+r+t)^b}
  \end{equation}
holds for any $(r,t)\in \Sigma$. We note that from Lemma \ref{lem_firs}, we see that this estimate holds with $a=1$, $b=1+\nu$, $c=A\e/8$ and $d=0$.

Under the assumption (\ref{basic-est}), we prove the following estimate for the convolution term $G_{\gamma}(u^2)$ in the right-hand side of  (\ref{positive}):
\begin{prop}
\label{prop:G-est}
Besides the assumptions of Proposition \ref{lem:positive}, we assume that the estimate (\ref{basic-est}) holds for any $(r,t)\in \Sigma$ and some $a,b,c,d$ and $\gamma\ge 0$. Let $(r,t)\in \Sigma$. Then the estimate
\begin{equation}
\label{G-est}
G_{\gamma}(u^2)(\lambda,s)\ge \frac{Cc^2s^{2a+\frac{3n-3}{2}}\left\{\lambda-s-\max(R,\delta s)\right\}^{2d+1}
}{(2d+1)2^{\gamma}\lambda^{\frac{n-1}{2}+\gamma}(1+s+\lambda)^{2b}}
\end{equation}
holds for any $\lambda\in [r-t+s,r+t-s]$ and $s\in [0,t]$, where $C=C(n)>0$ is a positive constant depending only on $n$.
\end{prop}

\begin{proof}[Proof of Proposition \ref{prop:G-est}]
Set $C_0:=2^{3-n}\omega_{n-1}$. Since $(r,t)\in \Sigma$ and the estimates $\lambda\ge r-t+s$ and $s\le t$ hold, the inequality $\lambda-s\ge \max(R,\delta s)$ holds. Thus by the identity (\ref{G}) and the assumptions (\ref{basic-est}) and $\gamma\ge 0$, the estimates
\begin{align}
    &G_{\gamma}(u^2)(\lambda,s)\notag\\
    &\ge \frac{C_0}{\lambda^{n-2}}\int_{s+\max(R,\delta s)}^{\infty} \rho u^2(\rho,s)
    \left\{\int_{|\rho-\lambda|}^{\rho+\lambda}\eta^{1-\gamma}h(\eta,\rho,\lambda)d\eta \right\}d\rho\notag\\
    &\ge \frac{C_0c^2s^{2a}}{\lambda^{n-2}}\int_{s+\max(R,\delta s)}^{\infty}\frac{\rho\left\{\rho-s-\max(R,\delta s)\right\}^{2d}}{(1+\rho+s)^{2b}}\left\{\int_{|\rho-\lambda|}^{\rho+\lambda}\eta^{1-\gamma}h(\eta,\rho,\lambda)d\eta \right\}d\rho\notag\\
    &\ge \frac{C_0c^2s^{2a}}{2^{\gamma}\lambda^{n-2+\gamma}(1+\lambda+s)^{2b}}
\int_{s+\max(R,\delta s)}^{\lambda}\rho\left\{\rho-s-\max(R,\delta s)\right\}^{2d}\notag\\
 &\ \ \ \ \ \ \ \ \ \ \ \ \ \ \ \ \ \ \ \ \ \ \ \ \ \ \ \ \ \ \ \ \ \ \ \ \ \ \ \ \ \ \ \ \ \ \ \ \ \ \ \times\left\{\int_{\lambda-\rho}^{\lambda+\rho}\eta h(\eta,\rho,\lambda)d\eta \right\}d\rho\notag\\
 &\ge \frac{C_0c^2s^{2a+1}}{2^{\gamma}\lambda^{n-2+\gamma}(1+\lambda+s)^{2b}}
\int_{s+\max(R,\delta s)}^{\lambda}\left\{\rho-s-\max(R,\delta s)\right\}^{2d}\notag\\
 &\ \ \ \ \ \ \ \ \ \ \ \ \ \ \ \ \ \ \ \ \ \ \ \ \ \ \ \ \ \ \ \ \ \ \ \ \ \ \ \ \ \ \ \ \ \ \ \ \ \ \ \times\left\{\int_{\lambda-\rho}^{\lambda+\rho}\eta h(\eta,\rho,\lambda)d\eta \right\}d\rho
 \label{3-9}
\end{align}
hold. In the following, we divide the proof into the two cases, i.e. $n\ge 3$ and $n=2$.

  \par\noindent
  {\bf $\cdot$ Case 1:\ $n\ge 3$}. In the case of $0\le s+\max (R,\delta s)\le \rho\le \lambda$ and $0\le \lambda-\rho\le \eta$, the estimates
\[
   \eta\ge \eta-\lambda+\rho,\ \ \rho+\lambda+\eta\ge \lambda,\ \ \eta-\rho+\lambda\ge \eta+\rho-\lambda
\]
hold, which implies the inequalities
\begin{align}
    \eta h(\eta,\rho,\lambda)
    &=\eta(\eta+\rho-\lambda)^{\frac{n-3}{2}}(\eta-\rho+\eta)^{\frac{n-3}{2}}(\rho+\lambda+\eta)^{\frac{n-3}{2}}(\rho+\lambda-\eta)^{\frac{n-3}{2}}\notag\\
    &\ge \lambda^{\frac{n-3}{2}}(\eta+\rho-\lambda)^{n-2}(\rho+\lambda-\eta)^{\frac{n-3}{2}}
    \label{3-10}
\end{align}
hold. We note that for any $p,q>0$, the identity is well known;
\begin{equation}
\label{3-11}
     B(p,q)=\frac{\Gamma(p)\Gamma(q)}{\Gamma(p+q)},
\end{equation}
where $B:\R_{\ge 0}\times \R_{\ge 0}\rightarrow \R_{\ge 0}$ is the Beta function given by
\[
    B(p,q):=\int_0^1y^{p-1}(1-y)^{q-1}dy
\]
and $\Gamma$ is the Gamma function given by (\ref{gamma}).
By the identity (\ref{3-11}) and changing variables, the identities
\begin{align*}
\label{3-12}
     \int_{\alpha}^{\beta}(z-\alpha)^{p-1}(\beta-z)^{q-1}dz=(\beta-\alpha)^{p+q-1}B(p,q)
     =(\beta-\alpha)^{p+q-1}\frac{\Gamma(p)\Gamma(q)}{\Gamma(p+q)}
\end{align*}
hold for any $\alpha<\beta$ and $p,q>1$. By these identities with $\alpha=\lambda-\rho$, $\beta=\lambda+\rho$, $p=\frac{n-1}{2}$ and $q=n-1$, the identity
\begin{equation}
\label{3-12}
      \int_{\lambda-\rho}^{\lambda+\rho}(\rho+\lambda-\eta)^{\frac{n-3}{2}}(\eta-\lambda+\rho)^{n-2}d\eta=(2\rho)^{\frac{3n-5}{2}}\frac{\Gamma(\frac{n-1}{2})\Gamma(n-1)}{\Gamma\left(\frac{3(n-1)}{2}\right)}=:C_1\rho^{\frac{3n-5}{2}}
\end{equation}
holds, where $C_1=C_1(n)>0$ is a constant depending only on $n$.
By combining the estimates (\ref{3-9}), (\ref{3-10}) and (\ref{3-12}), the estimates
  \begin{align}
  &G_{\gamma}(u^2)(\lambda,s)\notag\\
  &\ge \frac{C_0c^2s^{2a+1}}{2^{\gamma}\lambda^{\frac{n-1}{2}+\gamma}(1+\lambda+s)^{2b}}
\int_{s+\max(R,\delta s)}^{\lambda}\left\{\rho-s-\max(R,\delta s)\right\}^{2d}\notag\\
 &\ \ \ \ \ \ \ \ \ \ \ \ \ \ \ \ \ \ \ \ \ \ \ \ \ \ \times\left\{\int_{\lambda-\rho}^{\lambda+\rho}(\rho+\lambda-\eta)^{\frac{n-3}{2}}(\eta-\lambda+\rho)^{n-2}d\eta \right\}d\rho\notag\\
 &=\frac{C_0C_1c^2s^{2a+1}}{2^{\gamma}\lambda^{\frac{n-1}{2}+\gamma}(1+\lambda+s)^{2b}}\int_{s+\max(R,\delta s)}^{\lambda}\rho^{\frac{3n-5}{2}}\left\{\rho-s-\max(R,\delta s)\right\}^{2d}d\rho\notag\\
 &\ge \frac{C_0C_1c^2s^{2a+\frac{3(n-1)}{2}}}{2^{\gamma}\lambda^{\frac{n-1}{2}+\gamma}(1+\lambda+s)^{2b}}\int_{s+\max(R,\delta s)}^{\lambda}\left\{\rho-s-\max(R,\delta s)\right\}^{2d}d\rho\notag\\
 &=\frac{C_0C_1c^2s^{2a+\frac{3(n-1)}{2}}\left\{\lambda-s-\max(R,\delta s)\right\}^{2d+1}}{(2d+1)2^{\gamma}\lambda^{\frac{n-1}{2}+\gamma}(1+\lambda+s)^{2b}}
  \end{align}
hold, which implies (\ref{G-est}) with $C:=C_0C_1$.

    \par\noindent
    {\bf $\cdot$Case 2: $n=2$}. In the case of $0\le \rho\le \lambda$ and $\eta\in (\lambda-\rho,\lambda+\rho)$, the estimates
\[
   \eta^2-(\rho-\lambda)^2\le \eta^2\ \ \text{and}\ \ \ (\rho+\lambda+\eta)^{\frac{1}{2}}\le 2\lambda^{\frac{1}{2}}
\]
hold, which implies the inequalities
\[
    \eta h(\eta,\rho,\lambda)=\eta\left\{\eta^2-(\rho-\lambda)^2\right\}^{-\frac{1}{2}}(\rho+\lambda+\eta)^{-\frac{1}{2}}(\rho+\lambda-\eta)^{-\frac{1}{2}}\ge 2^{-1}\lambda^{-\frac{1}{2}}(\rho+\lambda-\eta)^{-\frac{1}{2}}
\]
hold. By these estimates, the inequalities
\[
    \int_{\lambda-\rho}^{\lambda+\rho}\eta h(\eta,\rho,\lambda)d\eta
    \ge 2^{-1}\lambda^{-\frac{1}{2}}\int_{\lambda-\rho}^{\lambda+\rho}(\rho+\lambda-\eta)^{-\frac{1}{2}}d\eta=\lambda^{-\frac{1}{2}}(2\rho)^{\frac{1}{2}}
\]
hold. By combining this estimate and (\ref{3-9}), the inequalities
\begin{align*}
    &G_{\gamma}(u^2)(\lambda,s)\\
    &\ge \frac{\sqrt{2}C_0c^2s^{2a+1}}{2^{\gamma}\lambda^{\gamma+\frac{1}{2}}(1+\lambda+s)^{2b}}
\int_{s+\max(R,\delta s)}^{\lambda}\rho^{\frac{1}{2}}\left\{\rho-s-\max(R,\delta s)\right\}^{2d}d\rho\\
&\ge \frac{\sqrt{2}C_0c^2s^{2a+\frac{3}{2}}}{2^{\gamma}\lambda^{\gamma+\frac{1}{2}}(1+\lambda+s)^{2b}}
\int_{s+\max(R,\delta s)}^{\lambda}\left\{\rho-s-\max(R,\delta s)\right\}^{2d}d\rho\\
&=\frac{\sqrt{2}C_0c^2s^{2a+\frac{3}{2}}\left\{\lambda-s-\max(R,\delta s)\right\}^{2d+1}}{(2d+1)2^{\gamma}\lambda^{\gamma+\frac{1}{2}}(1+\lambda+s)^{2b}}
\end{align*}
hold, which implies (\ref{G-est}) with $C:=\sqrt{2}C_0$, which completes the proof of the proposition.
\end{proof}

\subsection{Complete of the proof of Theorem \ref{T.1.1} in high spatial dimension $n\ge 2$}

Now we give a proof of Theorem \ref{T.1.1}.
\begin{proof}[Proof of Theorem \ref{T.1.1}]
Let $T>0$ and $u=u(r,t)\in C([0,\infty)\times [0,T))$ be a radial solution to (\ref{IE_u_depend}) with the data $(f,g)$ satisfying the assumptions in Theorem \ref{T.1.1}. Let $(r,t)\in \Sigma(T,R,\delta)$. We assume that the estimate (\ref{basic-est}) holds for any $(r,t)\in \Sigma$ and some $a,b,c,d$. Then by the assumptions on the data $(f,g)$, the positivity (\ref{positive-high}) of the solution and inserting the estimates (\ref{G-est}) and (\ref{basic-est}) into the second term of (\ref{positive}), the estimates
\begin{align*}
    u(r,t)&\ge \frac{Cc^3}{8(2d+1)r^m}\int_0^t\frac{s^{3a+\frac{3n-3}{2}}}{(1+s)^2}\left\{\int_r^{r+t-s}\frac{\lambda^{m-\frac{n-1}{2}-\gamma}\left\{\lambda-s-\max(R,\delta s)\right\}^{3d+1}}{(1+s+\lambda)^{3b}}d\lambda\right\}ds\\
    &\ge \frac{Cc^3\left\{r-t-\max(R,\delta t)\right\}^{3d+1}}{8(2d+1)(1+r+t)^{3b+\frac{n+3}{2}+\gamma}}\int_0^ts^{3a+\frac{3n-3}{2}}(t-s)ds\\
    &=\frac{Cc^3t^{3a+\frac{3n+1}{2}}\left\{r-t-\max(R,\delta t)\right\}^{3d+1}}{8(2d+1)\left\{3a+(3n+1)/2\right\}^2(1+r+t)^{3b+\frac{n+3}{2}+\gamma}}
\end{align*}
hold. We introduce the sequences $\{a_j\}_{j\in \N}$, $\{b_j\}_{j\in \N}$, $\{c_j\}_{j\in \N}$, $\{d_j\}_{j\in \N}$, which are defined by
\EQS{
\label{a_j}
  a_{j+1}&=3a_j+\frac{3n+1}{2}\quad (j\in\N), &&a_1:=1,\\
\label{b_j}
  b_{j+1}&=3b_j+\frac{n+3}{2}+\ga \quad (j\in \N), &&b_1:=1+\nu,\\
\label{c_j}
  c_{j+1}&=\frac{Cc_j^3}{8(2d_j+1)\{3a_j+(3n+1)/2\}^2}\quad (j\in\N), &&c_1:=\frac{A\e}{8},\\
\label{d_j}
  d_{j+1}&=3d_j+1\quad (j\in\N), &&d_1:=0.
}
By solving the difference equations (\ref{a_j}), (\ref{b_j}) and (\ref{d_j}), the identities
\EQS{
\label{res-seq1}
a_j&=3^{j-1}\left(\frac{3n+5}{4}\right)-\frac{3n+1}{4},\\
\label{res-seq2}
b_j&=3^{j-1}\left(\frac{7+4\nu+2\gamma+n}{4}\right)-\frac{2\gamma+n+3}{4},\\
\label{res-seq3}
d_{j}&=\frac{3^{j-1}}{2}-\frac{1}{2}
}
hold for any $j\in \N$. By these identities and the relation (\ref{c_j}), the estimate
 \begin{equation}
 \label{c_j2}
  c_{j+1}\ge \frac{Dc_j^3}{3^{3j}}
\end{equation}
holds for any $j\in \N$, where $D=D(n)>0$ is a constant given by
\[
   D:=\frac{6C}{(3n+5)^2}.
\]
By the estimate (\ref{c_j2}), the inequality
\[
    \log c_{j+1}-\frac{3}{2}(\log 3)(j+1)+\frac{1}{2}\log D-\frac{3}{4}\log 3\ge 3\left\{\log c_j-\frac{3}{2}(\log3)j+\frac{1}{2}\log D-\frac{3}{4}\log 3 \right\}
\]
holds for any $j\in \N$, which implies that the the estimate
\begin{equation}
\label{C_j-ind}
    c_j\ge D^{-\frac{1}{2}}\exp\left(3^{j-1}\log\left(c_13^{-\frac{9}{4}}D^{\frac{1}{2}}\right)\right)
\end{equation}
holds for any $j\in \N$.

Here we assume that the existence time $T$ satisfies
\[
     T>\max(R/\delta,1).
\]
Then we can define the following half line $\ell$ in the region $\Sigma(T,R,\delta)$, which is given by
\[
   \ell:=\left\{(r,t)\in (0,\infty)\times (0,T)|\ t\ge \max(R/\delta,1),\ r=2(1+\delta)t\right\}\subset \Sigma.
\]
Then for any $(r,t)\in \ell$, the estimates
\[
r-t-\max\{R,\delta t\}=(1+\delta)t\ge t , \ 1+r+t\le t+2(1+\delta)t+t=2t(2+\delta)
\]
hold. Here we remember the definition of $c_1$, i.e. $c_1:=A\varepsilon 2^{-3}$. By combining the estimates (\ref{basic-est}) with $a=a_j$, $b=b_j$, $c=c_j$ and $d=d_j$, (\ref{res-seq1}), (\ref{res-seq2}), (\ref{res-seq3}) and (\ref{C_j-ind}), the estimates
  \[
\begin{array}{lll}
  u(2(1+\delta)t,t)&>&D^{-1/2}\exp\left(3^{j-1}\left(\log (c_1 3^{-\frac{9}{4}}D^{1/2})\right)\right)
t^{3^{j-1}\left(\frac{n-\gamma-2\nu}{2}\right)}t^{\frac{\gamma-n}{2}}\\
&=:&\d D^{-1/2}\exp\left(3^{j-1}K(t)\right)t^{\frac{\gamma-n}{2}}
\end{array}
 \]
hold, for any $j\in \N$, where the function $K$ is defined by
  \[
 K(t):=\d \log\left(\e D^{1/2}A2^{-3}3^{-\frac{9}{4}}t^{\frac{n-\gamma-2\nu}{2}}\right)
 \]
for $\max(R/\delta,1)\le t\le T$.

Here we can define $B=B(n,\gamma,\nu,A)$ as
\[
B:=(D^{1/2}A2^{-3}3^{-\frac{9}{4}})^{-\frac{2}{n-\gamma-2\nu}}
\]
due to $n-\gamma-2\nu>0$. Moreover we can take $\e_0=\e_0(n,\gamma,R,g)>0$ such that the estimate
\[
B\e_0^{-\frac{2}{n-\gamma-2\nu}}\ge \max(R/\delta,1)
\]
holds.
On the contrary, for a fixed $\e\in (0,\e_0)$, we suppose that the lifespan $T_{\varepsilon}$ satisfies
\begin{equation}
\label{lifespan-1}
T_{\varepsilon}>B\e ^{-\frac{2}{n-\gamma-2\nu}}\ (\ge \max(R/\delta,1)).
\end{equation}
Then the estimate $K(T)>0$ holds for any $T\in \left(B\varepsilon^{-\frac{2}{n-\gamma-2\nu}},T_{\varepsilon}\right)$, which implies $u(2(1+\delta)T,T)\rightarrow \infty$ as  $j\rightarrow\infty$. This is a contradiction. Thus for any $\varepsilon\in (0,\varepsilon_0)$, the estimate $T_\e\le B\e ^{-2/(n-\gamma-2\nu)}$ holds, which completes the proof of Theorem \ref{T.1.1} for $n\ge2$.
\end{proof}

\section{Proof of Theorem \ref{T.1.1} in one spatial dimension}
\label{oned}
In this section, we give a proof of Theorem \ref{T.1.1} in one spatial dimension ($n=1$). To do so, we recall that the integral equation (\ref{IE_u_depend}) in one spatial dimension is written as
\begin{align*}
     u(x,t)&=\frac{\varepsilon}{2}\left[\left\{f(x+t)-f(x-t)\right\}+\int_{x-t}^{x+t}g(y)dy\right]\\
     &+\int^{t}_{0}\frac{1}{(1+s)^2}\left\{\int_{x-t+s}^{x+t-s}\left(\int_{\R}\frac{u^2(z,s)}{|y-z|^{\gamma}}dz\right)u(y,x)dy\right\}ds
\end{align*}
on $\R^n\times [0,T)$. Moreover if $f\equiv 0$, then the integral equation becomes
\begin{align}
\label{int-one}
u(x,t)=&\frac{\varepsilon}{2}\int_{x-t}^{x+t}g(y)dy
+\int^{t}_{0}\frac{1}{(1+s)^2}\left\{\int_{x-t+s}^{x+t-s}\left(\int_{\R}\frac{u^2(z,s)}{|y-z|^{\gamma}}dz\right)u(y,x)dy\right\}ds.
\end{align}
The following lemma means the positivity of the solution to (\ref{int-one}).
\begin{lem}[Positivity]
Let $\varepsilon>0$, $(f,g)\in C^1(\R)\times C^0(\R)$ satisfy the assumptions of Theorem \ref{T.1.1}, $T>0$ and $u=u(x,t)\in C(\R\times [0,T))$ be a solution to (\ref{IE_u_depend}). Then the estimate
\begin{equation}
\label{positive-one}
u>0\quad \mbox{in}\quad \{(x,t)\in(0,\infty)^2:  x-t\ge R\}
\end{equation}
holds.
\end{lem}

The positivity of $u$ follows from the comparison argument
by Keller\cite{K57}. For convenience of the readers, we give a proof of this lemma in Appendix.

Now we prove Theorem \ref{T.1.1}.

\begin{proof}[Proof of Theorem \ref{T.1.1} with $d=1$]
Let $T>0$ and $u\in C(\R^1\times[0,T))$ be a solution to (\ref{int-one}). By the assumption (\ref{asm-blowup2}) on the data $g$, the estimates
\begin{equation}
\label{4-2-a}
u(x,t)\ge \frac{A\e}{2}\int_{x-t}^{x+t}\frac{1}{(1+y)^{1+\nu}}dy\ge
\frac{A\e}{2}\int_{x}^{x+t}\frac{1}{(1+y)^{1+\nu}}dy
\ge\frac{A\e t}{2(1+x+t)^{1+\nu}}
\end{equation}
hold for any $(x,t)\in \R^n\times [0,T)$.

We assume that the following estimate
  \begin{equation}
  \label{basic-est1}
  u(x,t)\ge \frac{ct^a(x-t-R)^d}{(1+x+t)^b}
  \end{equation}
holds for any $(x,t)\in \R\times [0,T)$ with $x-t\ge R$, where $a,b=b(\gamma,\nu),c=c(A,\varepsilon),d$ are non-negative constants. From the estimate (\ref{4-2-a}), we see that the estimate (\ref{basic-est1}) holds with $a=1$, $b=1+\nu$, $c=A\e/2$, $d=0$.

Next we estimate the convolution term. For $(x,t)\in \R^n\times [0,T)$ with $x-t\ge R$, $s\in (0,t)$ and $y\in (x-t+s,x+t-s)$, the estimates
\[
    s+R\le x-t+s<y
\]
hold. Moreover in the case of $z\ge s+R$, by the assumption $\gamma>0$, the estimate
\[
    (y-z)^{\gamma}\le (1+y+s)^{\gamma}
\]
holds. By this estimate and the estimate (\ref{basic-est1}), the estimates
\begin{align}
\int_{\R}\frac{u^2(z,s)}{|y-z|^{\gamma}}dz&\ge
c^2s^{2a}\int_{s+R}^{y}\frac{(z-s-R)^{2d}}{(y-z)^{\gamma}(1+z+s)^{2b}}dz\notag\\
&\ge \frac{c^2s^{2a}}{(1+y+s)^{2b+\gamma}}\int_{s+R}^{y}(z-s-R)^{2d}dz\notag\\
&=\frac{c^2s^{2a}(y-s-R)^{2d+1}}{(2d+1)(1+y+s)^{2b+\gamma}}
\label{conv-est-one}
\end{align}
hold for any $s\in (0,t)$ and $y\in (x-t+s,x+t-s)$. Noting that the inequalities
\[
   1+s\le 1+t\le 2(1+x+t)
\]
hold for $(x,t)\in \R^n\times [0,T)$ with $x-t\ge R$
and $s\in (0,t)$, by combining the estimates (\ref{int-one}), (\ref{basic-est1}) and (\ref{conv-est-one}) and the assumption (\ref{asm-blowup2}) on the data $g$, the inequalities
\begin{align}
    u(x,t)&\ge \frac{c^3}{2d+1}\int_0^t\frac{s^{3a}}{(1+s)^2}\left\{\int_{x-t+s}^{x+t-s}\frac{(y-s-R)^{3d+1}}{(1+y+s)^{3b+\gamma}}dy\right\}ds\notag\\
    &\ge \frac{c^3(x-t-R)^{3d+1}}{4(2d+1)(1+x+t)^{3b+\gamma+2}}\int_0^ts^{3a}\left(\int_{x-t+s}^{x+t-s}dy\right)ds\notag\\
   &\ge \frac{c^3(x-t-R)^{3d+1}}{2(2d+1)(1+x+t)^{3b+\gamma+2}}\int_0^ts^{3a}(t-s)ds\notag\\
   &= \frac{c^3(x-t-R)^{3d+1}t^{3a+2}}{2(3a+1)(3a+2)(2d+1)(1+x+t)^{3b+\gamma+2}}
\end{align}
hold for any $(x,t)\in \R^n\times[0,T)$ with $x-t\ge R$.

Similarly to the argument of $n\ge2$, we next define the sequences like (\ref{a_j}) to (\ref{d_j}).
The sequences are same as by setting $n=1$ in (\ref{a_j}) to (\ref{d_j}).
with $c_1=A\e/2$.
Thus, we have
\[
\begin{array}{lll}
  u(2t+R,t)&>&CD^{-1/2}\exp\{3^{j-1}(\log (c_1 3^{-3s}D^{1/2}))\}
t^{3^{j-1}\{(1-\gamma-2\nu)/2\}}t^{(\gamma-1)/2}\\
&=&\d  CD^{-1/2}\exp\{3^{j-1}\wt{K}(t)\}t^{(\gamma-1)/2},
\end{array}
 \]
for $\max\{R,1\}\le t\le T$, where
 \[
 \wt{K}(t)=\d \log(\e D^{1/2}A2^{-3}3^{-3s}t^{(1-\gamma-2\nu)/2})
 \]
for $\max\{R/\delta,1\}\le t\le T$.
Similarly to the argument of $n\ge2$, there exists
a positive constant $\wt{B}=\wt{B}(\gamma,\nu)$ such that
$T_\e\le \wt{B}\e ^{-2/(1-\gamma-2\nu)}$ holds for small $\e$.
Thus the proof of Theorem \ref{T.1.1} is now completed for $n=1$.
\end{proof}



\section{Appendix}
\label{appendix1}
\ \ \ \
In this appendix, we give a proof of Proposition \ref{lem:positive}. In high spatial dimension ($n\ge 2$), we assume that
\begin{equation}
\label{asm-blow-up}
f\equiv0 \ \mbox{and} \ g \ \mbox{is radially symmetric if $n\ge2$}.
\end{equation}
Then the solution $u(\cdot,t)$ to (\ref{IE_u_depend}) is a radially symmetric function.
From Lemma 2.2 and Lemma 2.3 in \cite{T95}, we see that $u$ satisfies the following integral equation (\ref{int-odd}) or (\ref{int-even}).
When $n=2m+1$ $(m\in\N)$, we get
\begin{equation}
\label{int-odd}
u(r,t)=\frac{1}{2r^m}I(r,t,u_t(\cdot,0))+\frac{1}{2r^m}\int_{0}^{t}
I\left(r,t-s,\frac{G_{\gamma}(u^2)(\cdot,s)u(\cdot,s)}{(1+s)^2}\right)ds,
\end{equation}
where
\[
I(r,t,w(\cdot,\tau))=\int_{|r-t|}^{r+t}\lambda^mw(\lambda,\tau)
P_{m-1}\left(\frac{\lambda^2+r^2-t^2}{2r\lambda}\right)d\lambda.
\]
When $n=2m$ $(m\in\N)$, we get
\begin{equation}
\label{int-even}
u(r,t)=\frac{2}{\pi r^{m-1}}J(r,t,u_t(\cdot,0))+\frac{1}{\pi r^{m-1}}
\int_{0}^{t}J\left(r,t-s,\frac{G_{\gamma}(u^2)(\cdot,s)u(\cdot,s)}{(1+s)^2}\right)ds,
\end{equation}
where
\[
J(r,t,w(\cdot,s))=\int_{0}^{t}\frac{\rho d\rho}{\sqrt{t^2-\rho^2}}
\int_{|r-\rho|}^{r+\rho}\frac{\lambda^mw(\lambda,s)
T_{m-1}\left(\frac{\lambda^2+r^2-t^2}{2r\lambda}\right)}
{\sqrt{\lambda^2-(r-\rho)^2}\sqrt{(r+\rho)^2-\lambda^2}}d\lambda.
\]
Here $P_k$ and $T_k$ for $k\in\N\cup\{0\}$ denote the Legendre and Tschebyscheff polynomials of degree $k$ respectively, whose definitions can be seen in Lemmas 2.2 and 2.3 in \cite{T95}, respectively.

We prove Proposition \ref{lem:positive}. The argument
is similar to the one of Lemma 4.1 in \cite{TUW10}.
\begin{proof}[Proof of (\ref{positive-high}) and (\ref{positive-one}) in Proposition \ref{lem:positive} ]

Set $\delta=2/\delta_m$. Here $\delta_m$ is a positive constant
which satisfies
\[
P_{m-1}(z), \ T_{m-1}(z)\ge \frac{1}{2} \quad \mbox{for} \ \frac{1}{1+\delta_m}\le z \le 1.
\]
$P_m$ and $T_m$ are defined in (\ref{int-odd}) and (\ref{int-even}) respectively.
We define
\[
\Gamma(r,t)=\{(r,t)\in(0,\infty)^2:\ |r-\lambda|\le t-s\}.
\]
We note that $\Gamma(r_0,t_0)\subset \Sigma$ holds for any fixed $(r_0,t_0)\in\Sigma$.
We also set
\[
t_1=\inf\{t>0:\ u(r,t)=0, \ (r,t)\in\Gamma(r_0,t_0)\}.
\]
From the positivity assumption for $g(r)(=u_t(r,0)>0)$, we have $t_1>0$.

Assume that there exists $r_1>0$ such that $u(r_1,t_1)=0$ and $(r_1,t_1)\in \Gamma(r_0,t_0)$.
\vskip5pt
\par\noindent
{\bf Case $n=2m+1$.}

First of all, we note that
\[
\frac{\lambda^2+r^2-(t-s)^2}{2r\lambda}\ge \frac{(r-t+s)^2+r^2-(t-s)^2}{2r(r+t-s)}
=\frac{r-t+s}{r+t-s}\ge\frac{r-t}{r+t}
\]
holds for $(\lambda,s)\in\Gamma(r,t)$.
Thus, if $(r,t)\in\Sigma$, that is $r-t\ge (2/\delta_m)t$, then we have
\begin{equation}
\label{variable1}
\frac{r-t}{r+t}\ge \frac{1}{1+\delta_m}.
\end{equation}

By the definition of $t_1$, we have $u>0$ in $\Gamma(r_1,t_1)\setminus\{(r_1,t_1)\}$ which is the Duhamel term of (\ref{int-odd}) with $(r,t)=(r_1,t_1)$.
Then we have
\[
I(r_1,t_1-s,G_{\gamma}(u^2)(\cdot,s)u(\cdot,s))\ge\frac{1}{2}
\int_{r_1-t_1+s}^{r_1+t_1-s}\lambda^m G_{\gamma}(u^2)(\lambda,s)u(\lambda,s)d\lambda\ge0
\]
for $0\le s\le t_1$ by (\ref{variable1}).
It follows from $(r,t)=(r_1,t_1)$ in (\ref{int-odd}) that
\[
0=u(r_1,t_1)\ge\frac{1}{4r_1^m}\int_{r_1-t_1}^{r_1+t_1}\lambda^mu_t(\lambda,0)d\lambda>0,
\]
which leads contradiction. Therefore, we have $u>0$ in $\Sigma$.
\vskip5pt
\par\noindent
{\bf Case $n=2m$.}

Similarly to the case of $n=2m+1$, we note that
\[
\frac{\lambda^2+r^2-\rho^2}{2r\lambda}\ge\frac{r-\rho}{r+\rho}\ge \frac{r-t}{r+t}
\]
holds for $r-\rho\le \lambda\le r+\rho$ and $0\le\rho\le t$.
Similarly to the argument in the case of $n=2m+1$, we have
\[
\begin{array}{lllll}
J(r_1,t_1-s,G_{\gamma}(u^2)(\cdot,s)u(\cdot,s))\\
\d \ge \frac{1}{2}\int_{0}^{t_1-s}\frac{\rho d\rho}{\sqrt{(t_1-s)^2-\rho^2}}\int_{r_1-\rho}^{r_1+\rho}
\frac{\lambda^mG_{\gamma}(u^2)(\lambda,s)u(\lambda,s)d\lambda}{\sqrt{\lambda^2-(r_1-\rho)^2}
\sqrt{(r_1+\rho)^2-\lambda^2}}\ge0
\end{array}
\]
for $0\le s\le t_1$ by (\ref{variable1}).
Therefore, we obtain the following contradiction:
\[
\begin{array}{lllll}
0&=&u(r_1,t_1)\\
\d &\ge&\d\frac{1}{\pi r_1^{m-1}}
\int_{0}^{t_1}\frac{\rho d\rho}{\sqrt{t_1^2-\rho^2}}\int_{r_1-\rho}^{r_1+\rho}
\frac{\lambda^mu_t(\lambda,0)d\lambda}{\sqrt{\lambda^2-(r_1-\rho)^2}
\sqrt{(r_1+\rho)^2-\lambda^2}}>0.
\end{array}
\]
Therefore, we have $u>0$ in $\Sigma$.
In the case of one dimensional case, we note that (\ref{sol-one}) is
positive, so we do not have to assume the condition like $r-t\ge (2/\delta_m)t$.
The proof is the same as the one of the above argument, we omit its proof.
\end{proof}

\section*{Acknowledgement}
\par
The first author is supported by JST CREST Grant Number JPMJCR1913, Japan.
The first and third authors have been partially supported by
the Grant-in-Aid for Scientific Research (B) (No.18H01132) and Young Scientists Research (No.19K14581), Japan Society for the Promotion of Science.
The second author has been supported by RIKEN Junior Research Associate Program.
The third author has been partially supported
by the Grant-in-Aid for Scientific Research (B) (No.19H01795), Japan Society for the Promotion of Science.


\bibliographystyle{plain}

\begin{thebibliography}{20}



\bibitem{Asa86}{F. Asakura},
{\it Existence of a global solution to a semi-linear wave equation with slowly decaying initial data in three space dimensions},
Comm. Partial Differential Equations {\bf 11} (13) (1986) 1459--1487.
\bibitem{D15}{M. D'Abbicco},
{\it The threshold of effective damping for semilinear wave equations},
Math. Methods Appl. Sci. {\bf 38} (2015), 1032--1045.

\bibitem{DLR15}{M. D'Abbicco, S. Lucente and M. Reissig},
{\it A shift in the Strauss exponent for semilinear wave equations with a not effective damping},
J. Differential Equations {\bf 259} (2015), 5040--5073.

\bibitem{H20}{K. Hidano},
{\it Small data scattering and blow-up for a wave equation with a cubic convolution}, Funkcial. Ekvac. {\bf 43} (2000), 559--588.

\bibitem{IkeSoba}{M. Ikeda and M. Sobajima},
{\it Life-span of solutions to semilinear wave equation
with time-dependent critical damping
for specially localized initial data},
Math. Ann. {\bf 372} (2018), 1017--1040.

\bibitem{ISW1}{M. Ikeda, M. Sobajima and K. Wakasa},
{\it Test function method for blow-up phenomena of semi linear wave equations and their weakly coupled sysytems},
J. Differential Equations, {\bf 267} (2019), 5165--5201.

\bibitem{ISW}{M. Ikeda, M. Sobajima and Y. Wakasugi},
{\it Sharp lifespan estimates of blowup solutions to semilinear wave equations with time-dependent effective damping},
J. Hyperbolic Differ. Equ. {\bf 11} (2019), 795--819.

\bibitem{J55}{F. John},
\lq\lq Plane Waves and Spherical Means,
Applied to Partial Differential Equations",
Interscience Publishers, Inc., New York, 1955.

\bibitem{J79}{F. John},
{\it Blow-up of solutions of nonlinear wave equations in three space dimensions},
Manuscripta Math., {\bf 28} (1979), 235--268.

\bibitem{KT}{P. Karageorgis and K. Tsutaya},
{\it Existence and blowup for a Hartree-type wave equation}, in preparation.


\bibitem{K57}{J. B. Keller},
{\it On solutions of nonlinear wave equations},
Comm Pure Appl. Math.,{\bf 10} (1957), 523--530.

\bibitem{K96}{H. Kubo},
{\it On the critical decay and power for semilinear wave equations in odd space dimensions}, Discrete Contin. Dyn. Syst. {\bf 2} (1996) 173--190.

\bibitem{K04}{H. Kubo},
{\it On Point-Wise Decay Estimates for the Wave Equation and Their Applications},
Dispersive nonlinear problems in mathematical physics, 123--148, Quad. Mat., 15, Dept.
Math., Seconda Univ. Napoli, Caserta, 2004.

\bibitem{KK95}{H. Kubo and K. Kubota},
{\it Asymptotic behavior of radially symmetric solutions of $\Box u=|u|^p$ for super critical values $p$ in odd space dimensions}, Hokkaido Math. J. {\bf 24} (1995) 287--336.

\bibitem{KK98}{H. Kubo and K. Kubota},
{\it Asymptotic behavior of radially symmetric solutions of $\Box u=|u|^p$ for super critical values $p$ in even space dimensions}, Jpn. J. Math. {\bf 24} (1998) 191--256.

\bibitem{K93}{K. Kubota},
{\it Existence of a global solution to a semi-linear wave equation with initial data of non-compact support in low space dimensions}, Hokkaido Math. J. {\bf 22} (1993) 123--180.



\bibitem{LTW17}{N.-A. Lai, H. Takamura and K. Wakasa},
{\it Blow-up for semilinear wave equations
with the scale invariant damping and super-Fujita exponent},
J. Differential Equations {\bf 263} (2017), 5377--5394.

\bibitem{MS82}{G. P. Menzala and W. A. Strauss},
{\it On a wave equation with a cubic convolution}, J. Differential Equations {\bf 43} (1982), 93--105.

\bibitem{T95}{H. Takamura},
{\it  Blow-up for semilinear wave equations with slowly decaying data in high dimensions}, Differential Integral Equations {\bf 8} (1995) 647--661.


\bibitem{TUW10}{H. Takamura, H. Uesaka and K. Wakasa},
{\it Blow-up theorem for semilinear wave equations
with non-zero initial position},
J. Differential Equations {\bf 249} (2010) 914--930.

\bibitem{TW14}{H. Takamura and K. Wakasa},
{\it Almost global solutions of semilinear wave equations with
the critical exponent in high dimensions},
Nonlinear Analysis, TMA, {\bf 109} (2014), 187--229.



\bibitem{T92}{K. Tsutaya},
{\it A global existence theorem for semilinear wave equations with data of non compact support in two space dimensions},
Comm. Partial Differential Equations {\bf 17}, (1992) 1925--1954.

\bibitem{T93}{K. Tsutaya},
{\it Global existence theorem for semilinear wave equations with non-compact data in two space dimensions},
J. Differential Equations {\bf 104} (1993) 332-360.

\bibitem{T94}{K. Tsutaya},
{\it Global existence and the lifespan of solutions of semilinear wave equations with data of non compact support in three space dimensions}, Funkcial. Ekvac. {\bf 37} (1994) 1--18.

\bibitem{T03}{K. Tsutaya},
{\it Global existence and blow up for a wave equation with a potential and a cubic convolution}, Nonlinear Analysis and Applications: to V. Lakshmikantham on his 80th
Birthday. Vol. 1, 2, 913--937, Kluwer Acad. Publ., Dordrecht, 2003.

\bibitem{T14}{K. Tsutaya},
{\it Weighted estimates for a convolution appearing in the wave equation of Hartree type},
J. Math. Anal. Appl. {\bf 411} (2014), 719--731.

\bibitem{TL1}{Z. Tu and J. Lin},
{\it A note on the blowup of scale invariant damping wave equation
with sub-Strauss exponent},
arXiv:1709.00866.

\bibitem{TL2}{Z. Tu and J. Lin},
{\it Life-span of semilinear wave equations
with scale-invariant damping: critical Strauss exponent case},
Differential Integral Equations {\bf 32} (2019), 249--264.


\bibitem{Wa16}{K. Wakasa},
{\it The lifespan of solutions to semilinear damped wave equations in one space dimension},
Comm. Pure Appl. Anal. {\bf 15} (2016), 1265--1283.


\bibitem{Wakasugi14}{Y. Wakasugi},
{\it Critical exponent for the semilinear wave equation with scale invariant damping},
Fourier analysis, 375--390, Trends Math., Birkh\"{a}user/Springer, Cham, 2014.

\end{thebibliography}

\end{document}